\documentclass[oneside]{article}

\usepackage{amsmath, amsthm}
\usepackage{tikz}
\usepackage{pgfplots}
\usepackage{filecontents}
\usepackage{hyperref}

\theoremstyle{theorem}
\newtheorem{theorem}{Theorem}
 
\theoremstyle{definition}

\newcommand{\red}[1]{{\color{purple}{#1}}}
\newcommand{\blue}[1]{{\color{blue}{#1}}}

\allowdisplaybreaks

\begin{document}

\title{Stuttering look and say sequences and a challenger to Conway's most complicated algebraic number from the silliest source}
\markright{Stuttering Look and Say Sequences}
\author{Jonathan Comes}


\maketitle

\begin{abstract}
    We introduce stuttering look and say sequences and describe their chemical 
    structure in the spirit of Conway's work on audioactive decay. We show the 
    growth rate of a stuttering look and say sequence is an algebraic integer of degree 415.
\end{abstract}

\section{Introduction}

The following is an example of a \emph{look and say sequence}:
\begin{equation}\label{equation: standard las 111}
    111\to 31\to 1311\to 111321\to 31131211\to 132113111221\to\cdots
\end{equation}
Roughly speaking, the rule for generating the sequence is ``say what you see''. More precisely, the sequence above starts with the \emph{seed} 111. When we look at 111 we see \red{three} \blue{1}'s and thus the next term is \red{3}\blue{1}. Looking at 31 we see \red{one} \blue{3} followed by \red{one} \blue{1}, so the next term is \red{1}\blue{3}\red{1}\blue{1}. Continuing, from 1311 we see \red{one} \blue{1}, \red{one} \blue{3}, and then \red{two} \blue{1}'s, which gives the next term \red{1}\blue{1}\red{1}\blue{3}\red{2}\blue{1}.

John H.~Conway completely described the mathematical structure of look and say sequences in \cite{Conway}. In particular, Conway determined the growth rate of every look and say sequence. More precisely, he showed the ratios of the number of digits in successive terms in almost all\footnote{The only exceptions are the look and say sequences with the empty seed and the seed 22.} look and say sequences approach what is now known as \emph{Conway's constant}: 
\begin{equation}\label{Conway constant}
    \lambda = 1.303577269\ldots
\end{equation}
In other words, eventually each term in a look and say sequence like (\ref{equation: standard las 111}) will be approximately 30\% longer than its predecessor. Remarkably, Conway found $\lambda$ as the maximal real root of an irreducible degree 71 polynomial. In 1997 Conway remarked \emph{``To my mind, $\lambda$ is the current record holder for finding the most complicated algebraic number from the silliest source. I hope more math-funsters will get into this stuff and produce other challengers!''}\footnote{The quote can be found at \url{http://www.mathematik.uni-bielefeld.de/~sillke/SEQUENCES/series000}.}

\subsection{Stuttering look and say sequences}
In this article we produce such a challenger by looking at a variation of look and say sequences that we call \emph{stuttering look and say sequences}. To describe these sequences, it will be convenient to recall Conway's exponent notation from \cite{Conway}. Given any digit $d$ we write $d^n$ for the concatenation of $n$ copies of $d$. For example $2^3=222$. We use this exponent notation on runs of digits within a larger string of digits so that, for example, the sequence (\ref{equation: standard las 111}) can be written as 
\begin{equation*}
    1^3\to 31\to 131^2\to 1^3321\to 31^23121^2\to 1321^231^32^21\to\cdots
\end{equation*}
With this notation, the say what you see rule on a run of $n$ copies of a digit $d$ is merely $d^n\to nd$. Our stuttering look and say sequences are obtained by replacing this rule with $d^n\to n^nd$. For example, $222=2^3\to 3^32=3332$. In other words, when we look at $222$ and see \red{three} \blue{2}'s, we stutter the \red{three} (but not the \blue{2}) \red{three} times to obtain \red{333}\blue{2}. Here is a larger example: 
\[2222551=2^45^21^1\to 4^422^251^11=4444222511.\] 
Repeatedly applying this stuttering say what you see rule results in a stuttering look and say sequence. For example, here is the stuttering look and say sequence starting with seed 0:
\begin{equation}\label{stuttering sequence with seed 0}
    0\to 10\to 1110\to 333110\to 333322110\to 4444322222110\to\cdots 
\end{equation}
Compare this to the standard look and say sequence with seed 0:
\begin{equation}\label{std sequence with seed 0}
    0\to 10\to 1110\to 3110\to 132110\to 1113122110\to\cdots 
\end{equation}
The terms of (\ref{stuttering sequence with seed 0}) are growing in length faster than those in (\ref{std sequence with seed 0}). Indeed, this is evident in the following where the ratio of the number of digits in successive terms is plotted for the first 55 pairs of successive terms in both sequences. 
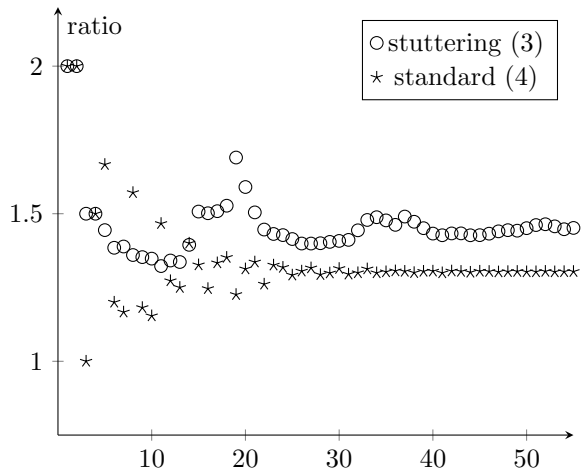
\[\begin{tikzpicture}
\begin{axis}[
    enlargelimits=true,
    axis lines=center,
    xmin = 0,
    xmax = 55,
    ymin = 0.75,
    ymax = 2.2,
    xlabel = {},
    ylabel = {ratio},
    legend pos = north east,
    colormap/blackwhite
]
\addplot+[
    only marks,
    scatter,
    mark=o,
    point meta=0,
    color=black,
    mark size=2.4pt,
    mark options={fill=white}
    ]
table{stutter_ratios.dat};
\addplot+[
    only marks,
    scatter,
    mark=star,
    point meta=0,
    color = black,
    mark size=2.4pt]
table{std_ratios.dat};
\legend{stuttering (\ref{stuttering sequence with seed 0}), standard (\ref{std sequence with seed 0})}
\end{axis}
\end{tikzpicture}\]
Notice that the ratios for the standard look and say sequence are indeed approaching Conway's constant (\ref{Conway constant}). On the other hand, the ratios for the stuttering look and say sequence appear to be approaching a number just less than 1.5. The main result of this article is an explicit description of this limit: 

\begin{theorem}\label{Main Result}
    The ratios of the number of digits in successive terms in the stuttering look and say sequence (\ref{stuttering sequence with seed 0}) approach
    \begin{equation}\label{stuttering Conway constant}
        \lambda = 1.4453300117\ldots 
    \end{equation}
    which is the maximal real root of the irreducible degree 415 polynomial given in Appendix \ref{appendix: the polynomial}. 
\end{theorem}

The proof of Theorem \ref{Main Result} is explained in Section \ref{section::chemistry}. The argument mimics the analogous proof given by Conway in \cite{Conway}. Indeed, we will show how the terms of a stuttering look and say sequence \emph{split} into compounds of substrings called \emph{elements}. The terms of (\ref{stuttering sequence with seed 0}) are eventually all compounds of 714 specific substrings, the \emph{common elements}. The long term behavior of (\ref{stuttering sequence with seed 0}) is completely determined by the \emph{chemistry} of these common elements, which can be summarized by a $714\times714$ \emph{decay matrix}. The growth rate (\ref{stuttering Conway constant}) is found as the maximal real eigenvalue of this decay matrix. 

The computations required to find the 714 common elements and their properties were worked out using a Python module called \texttt{look\_and\_say} (see Appendix \ref{appendix::python}). The \texttt{look\_and\_say} module, available at \url{https://pypi.org/project/look-and-say/}, was written to aid in the exploration of various look and say sequences. The interested reader may find this module useful in a search for more challengers!

\subsection{Decimal convention and other bases}

It is important to note that our stuttering look and say sequences rely on the standard decimal representations of natural numbers. This is contrary to Conway's convention in \cite{Conway} where every natural number is considered a \emph{digit}. Indeed, consider sequences with a seed consisting of ten 1's. In \cite{Conway}, Conway would use commas to specify the terms in this look and say sequence:
\begin{equation*}
1,1,1,1,1,1,1,1,1,1\to 10,1 \to 1, 10, 1, 1\to 1, 1, 1, 10, 2, 1\to \cdots
\end{equation*}
Alternatively, one can rely on decimal representations and consider only the digits $0, 1, 2,\ldots, 9$. With this convention, there is no need for commas, and the seed of ten 1's leads to the following look and say sequence:
\begin{equation*}
1111111111\to 101 \to 111011\to 311021\to\cdots
\end{equation*}
While the two conventions clearly lead to different look and say sequences, it turns out that the resulting chemistries are quite similar. This is because the structure of a generic look and say sequence is governed by Conway's 92 common elements, each of which contains only the digits 1, 2, and 3. In particular, the growth rates of look and say sequences using either convention are given by Conway's constant (\ref{Conway constant}). In fact, Conway's constant gives the generic growth rate of look and say sequences using standard base $b$ representations of integers whenever $b$ is greater than three\footnote{The structure of standard binary and ternary look and say sequences are described nicely by N.~Johnston in \url{http://www.nathanieljohnston.com/2010/11/the-binary-look-and-say-sequence/} and \url{http://www.njohnston.ca/2011/01/further-variants-of-the-look-and-say-sequence/}.}. 

Our stuttering look and say sequences use decimal representations with digits $0, 1, 2,\ldots 9$. For example, here is our stuttering sequence with a seed $1^{10}$:
\begin{equation}\label{stuttering seed ten 1s}
1^{10} \to 101010101010101010101 \to 1^301^301^301^301^301^301^301^301^301^301^2\to\cdots
\end{equation}
In the stuttering case, the chemistry does depend on this convention. In particular, Theorem \ref{Main Result} does not hold when all natural numbers are considered digits. The method we use to prove Theorem \ref{Main Result} breaks down when we consider every natural number as a digit since as all terms do not split into a finite collection of common elements. On the other hand, the proof described in Section \ref{section::chemistry} adapts easily when we consider standard number systems using other bases (binary, ternary, etc.). The following table lists the algebraic integer $\lambda$ (the growth rate), the degree\footnote{The Python session used to compute the values in the table failed to produce a result for the degree of $\lambda$ when $b=12$.} of $\lambda$, and the number of common elements for the stuttering look and say sequence generated by the seed 0 using standard base $b$ representations of integers.
\[\begin{array}{|c|c|c|c|}
    \hline
    b  & \lambda & \text{degree} & \text{no. elements}\\\hline
    2  & 2.8923039932\ldots & 4   & 5  \\
    3  & 2.0062263631\ldots & 8   & 15 \\
    4  & 1.8785595146\ldots & 25  & 82 \\
    5  & 1.6435823791\ldots & 46  & 96 \\
    6  & 1.6042539483\ldots & 221 & 364\\
    7  & 1.5054743588\ldots & 220 & 422\\
    8  & 1.5106915019\ldots & 310 & 664\\
    9  & 1.4376480767\ldots & 312 & 575\\
    10 & 1.4453300117\ldots & 415 & 714\\
    11 & 1.3952766103\ldots & 477 & 859\\
    12 & 1.3988336699\ldots &     & 1043\\
    \hline
\end{array}\] 
Note the degree of $\lambda$ tends to increase with $b$, suggesting that there is not a \emph{most} complicated algebraic integer from this silly source. 

\subsection{Other versions of stuttering look and say sequences} In \cite{other-stutter} the authors consider what they call ``stuttering Conway sequences'' which are closely related to our stuttering look and say sequences. Their version of the stuttering say what you see rule is $d^n\to n^jd^j$ for some fixed natural number $j$. They show that when $j\in\{2,3\}$ these stuttering Conway sequences are equivalent to Conway's look and say sequences. For $j>3$ the Python \texttt{look\_and\_say} module can be used to determine the structure of these sequences after determining appropriate splitting conditions. In particular, the $\lambda$'s (i.e.~growth rates) for these sequences can be found as algebraic integers of various degrees. Initial investigations show the degrees of these $\lambda$'s are relatively low:
\[\begin{array}{|c|ccccccccc|}
    \hline
    j & 2 & 3 & 4 & 5 & 6 & 7 & 8 & 9 & 10\\\hline
    \text{degree} & 71 & 71 & 22 & 34 & 7 & 37 & 18 & 18 & 12\\
    \hline
\end{array}\]

\section{The Chemistry}\label{section::chemistry}

In this section we give more details on the proof of Theorem \ref{Main Result}. In particular, we describe the chemistry of the stuttering look and say sequences in the spirit of Conway's \cite{Conway}. We then briefly discuss stuttering sequences starting with seeds other than 0, culminating in a cosmological conjecture that implies all stuttering look and say sequences grow at the same rate (\ref{stuttering Conway constant}). To fully understand the details, the reader should be familiar with \cite{Conway}. We will not attempt to further summarize the results in \cite{Conway} since we cannot hope to improve upon Conway's delightful treatment. 

Moving forward it will be useful to introduce a bit more notation. Assume $S$ is a string of digits. We extend Conway's exponent notation and write $S^n$ for the concatenation of $n$ copies of $S$. For example, $3321112111=3^221^321^3=3^2(21^3)^2$.
Moreover, we use Conway's subscript notation by letting $S_n$ denote the $n$th descendant of $S$. In other words, the stuttering look and say sequence with seed $S=S_0$ will be denoted
$S_0\to S_1\to S_2\to S_3\to\cdots$.

\subsection{Splitting into elements} 
Following \cite{Conway} we say a string $S=LR$ \emph{splits (into a compound of $L$ and $R$)}  
if $S_n=L_nR_n$ for all $n\geq 0$. We have the following (simpler) analog of Conway's Splitting Theorem:

\begin{theorem}\label{stuttering splitting theorem}
The string $S=LR$ splits into a compound of $L$ and $R$ whenever the last digit in $L$ is 0 and the first digit in $R$ is not 0.
\end{theorem}

\begin{proof}
If the last digit in $L$ is 0, then the last digit in $L_n$ will be 0 for all $n\geq 0$. On the other hand, the first digit in $R_n$ will always be the leading digit of some decimal representation of a natural number, so the first digit in $R_n$ will never be 0. It follows that the stuttering say what you see operation on $L_nR_n$ can be applied to $L_n$ and $R_n$ separately. Now, a simple induction argument shows $S_n=L_nR_n$ for all $n\geq 0$.
\end{proof}

An \emph{element} is a string that does not split into a compound of smaller strings. By Theorem \ref{stuttering splitting theorem} we know a string is an element whenever the string has at most one run of 0's, and that run appears at the end of the string. For example, the second term of (\ref{stuttering seed ten 1s}) splits into a compound of eleven elements: ten $10$'s and one $1$. The third term of (\ref{stuttering seed ten 1s}) splits into a compound of ten $1^30$'s and one $1^2$. 
Applying the stuttering say what you see operation to a string $S$ and splitting the result yields a compound of elements which we call the \emph{decay} of $S$.  
For example, since $0^{10}\to (10)^{10}0=(10)^{9}10^2$ we say $0^{10}$ decays into a compound of nine $10$'s and one $10^2$. 

\subsection{The common elements} 
Each of the six terms of the stuttering look and say sequence starting with seed 0 shown in (\ref{stuttering sequence with seed 0}) do not split. This continues for another eight terms where we find the last non-splitting term:
{\small\begin{equation}\label{last non-split term}
8^871^363^42^21^254^513^42^316^651^343^31^236^653^72^31^238^8173^31^264^432^51^259^981^379^{10}5^54^53^42^31^20
\end{equation}}
Thanks to the $9^{10}$, the element above decays into a compound of eleven elements, nine of which are $10$, which is the second term of (\ref{stuttering sequence with seed 0}). It follows that every element appearing in a term of (\ref{stuttering sequence with seed 0}) after the seed will appear infinitely often in (\ref{stuttering sequence with seed 0}). We will call these elements the \emph{common elements}. The set $C$ of all common elements can be found as follows: let $C_0=\{10\}$ and for $i>0$ let $C_i$ denote the set obtained from $C_{i-1}$ by adding to it all elements $e\not\in C_{i-1}$ such that $e$ is in the decay of some element of $C_{i-1}$. If $C_{i+1}=C_i$ for some $i$, then $C=C_i$ is finite. This is essentially the algorithm for finding common elements implemented in the Python \texttt{look\_and\_say} module, which found that there are exactly 714 common elements (see Appendix \ref{appendix::python}). 

We denote the common elements by $e_1,\ldots,e_{714}$. In particular, we set $e_1=10$. We will not give the explicit definition of each $e_i$ as a string of digits because it would take up too much space (the longest element, $e_{619}$, has 45,460 digits!). However, one can recover the strings recursively using the periodic table found in Appendix \ref{appendix::periodic table}, which lists the decay of each of the common elements. For example, the first entry in the periodic table states that $e_1$ decays into $e_2$. Since $e_1=10\to1110$, it follows that $e_2=1110$. Using the periodic table we can write the sequence (\ref{stuttering sequence with seed 0}) in terms of elements:
\begin{align*}
0   & \to e_{1}\to e_{2}\to e_{3}\to e_{4}\to e_{5}\to e_{6}\to e_{7}\to e_{8}\to e_{9}\to e_{10}\to e_{11}\to e_{13}\\
    & \to e_{16}\to e_{20}e_1^9e_{17}\to e_{25}e_1^9e_{12}e_2^9e_{21}\to e_{33}e_1^9e_{14}e_2^9e_{15}e_3^9e_{28}\to\cdots
\end{align*}
In particular, since the string (\ref{last non-split term}) is the last non-splitting term of (\ref{stuttering sequence with seed 0}), that string is $e_{16}$. 
Notice $e_1,\ldots,e_{714}$ are not labeled in the order they appear in (\ref{stuttering sequence with seed 0}). Instead, the elements are ordered according to their relative abundances, which will be discussed below.

\subsection{The decay matrix}
Let $D$ denote the $714\times 714$ matrix whose $(i,j)$-entry is the number of times $e_i$ occurs in the decay of $e_j$. The entries in $D$ can be read off the periodic table. For example, according to the periodic table 
\[e_{335}\to e_{366}(e_1^9e_{26})^{175}e_1^9e_{365}e_1^9e_{293}.\]
Counting the $e_1$'s in the expression above we see the $(1,335)$-entry in $D$ is 1593. This is the maximum entry in $D$, and it occurs twice: the $(1,619)$-entry is also 1593. 

As in \cite{Conway}, $D$ is the key to describing the long term behavior of (\ref{stuttering sequence with seed 0}). In particular, a straightforward modification of Conway's argument shows the growth rate of (\ref{stuttering sequence with seed 0}) is given by the dominant (i.e.~maximal real) eigenvalue of $D$. A Python session found this eigenvalue to be (\ref{stuttering Conway constant}) (see Appendix \ref{appendix::python}). As $\lambda$ is an eigenvalue, it is a root of the characteristic polynomial of $D$, which factors as 
\begin{equation*}
\lambda^{283} \left(\lambda - 1\right) \left(\lambda + 1\right) \left(\lambda^{2} + 1\right) \left(\lambda^{4} + 1\right) \left(\lambda^{8} + 1\right) p(\lambda)
\end{equation*}
where $p(\lambda)$ is the 415 degree polynomial given in Appendix \ref{appendix: the polynomial} (see Appendix \ref{appendix::python}). Since (\ref{stuttering Conway constant}) is a not a root of the other factors, it must be a root of $p(\lambda)$. This completes the summary of the proof of Theorem \ref{Main Result}.

An eigenvector of $D$ with eigenvalue (\ref{stuttering Conway constant}) will give the limiting relative abundances of the common elements appearing in (\ref{stuttering sequence with seed 0}). For the sake of space these abundances are not included in the periodic table, but they can be reproduced using the Python \texttt{look\_and\_say} module as shown in Appendix \ref{appendix::python}. As mentioned above, these abundances were used to order the elements. The following table gives the limiting relative abundances for the first three and last of the elements:
\[\begin{array}{|c|ccccc|}
    \hline
    \text{element} & e_1 & e_2 & e_3 & \cdots & e_{714}\\\hline
    \text{abundance (\%)} & 27.81585668 &
                            19.24533252 &
                            13.31552819 &
                            \cdots &
                             0.00000004\\
    \hline
\end{array}\]

\subsection{End elements and a cosmological conjecture} If one focuses on only the last element appearing in each term of (\ref{stuttering sequence with seed 0}), then 26 terms after the seed one will find 
\[e_{244}=9^{10}8^97^86^75^65^44^53^42^31^20.\]
Since $e_{244}\to e_1^{10}e_{244}$, every term in (\ref{stuttering sequence with seed 0}) after this point will also end with $e_{244}$. Now, consider the stuttering look and say sequences starting with a digit seed $d\not=0$. If $d\not=1$ the sequence is the same as (\ref{stuttering sequence with seed 0}) with the exception of the very last digit:
\begin{equation}\label{stutter sequence with seed d not 1}
    d\to 1d\to 111d\to 33311d\to 33332211d\to 444432222211d\to\cdots
\end{equation}
Thus, after 26 terms the only element that appears in (\ref{stutter sequence with seed d not 1}) that is not among the 714 common elements is the following isotope of $e_{244}$:
\begin{equation}\label{isotope of e244}
    9^{10}8^97^86^75^65^44^53^42^31^2d\qquad (2\leq d\leq 9).
\end{equation}
On the other hand, the sequence starting with seed 1 is a bit different:
\begin{equation}\label{stutter sequence with seed 1}
    1\to 11\to 221\to 22211\to 3332221\to 3333333211\to7777777312221\to\cdots
\end{equation}
The 16th and 17th terms of (\ref{stutter sequence with seed 1}) both split into 21 elements, the last elements of each we denote by $\varepsilon_1$ and $\varepsilon_2$ respectively. These two elements are given below:
\begin{align*}
    \varepsilon_1 = &~ 91^343^31^232^31^24^431^23^321^27^761^353^42^215^5413^42^319^94131^23^321^26^651^343^42^2\\
                    &~ 14^43^421^25^541^332^31^24^431^23^321^2, \\
    \varepsilon_2 = &~ 193^31143^42^21^23^42^314^5132^213^412^317^8163^31^254^432^31^25^6141^24^43^421^29^{10}14\\
                    &~ 1^332^213^412^316^7153^31^24^532^31^24^9312^315^6143^31^23^42^314^5132^213^412^31.
\end{align*}
These elements decay as $\varepsilon_1\to \varepsilon_2\to e_{43}e_1^{9}\varepsilon_1$. Thus, the last element in the terms of (\ref{stutter sequence with seed 1}) eventually alternate between $\varepsilon_1$ and $\varepsilon_2$. Since $\varepsilon_1,\varepsilon_2$, and the isotopes of $e_{244}$ given in (\ref{isotope of e244}) do not end with a 0, they can only appear as the last element in a compound. In particular, these elements will all have relative abundance 0. A Python session similar to that in Appendix \ref{appendix::python} shows that eventually each term of (\ref{stutter sequence with seed 1}) splits into a compound of the 714 common elements along with $\varepsilon_1$ or $\varepsilon_2$. It follows that all ten stuttering look and say sequences with seed $d\in\{0,1,2,\ldots,9\}$ grow at rate (\ref{stuttering Conway constant}).

We conjecture that in \emph{every} stuttering look and say sequence (i.e. starting with any nonempty string of digits as a seed) the terms eventually all split into compounds of the 714 common elements, the isotopes of $e_{244}$ in (\ref{isotope of e244}), $\varepsilon_1$, and $\varepsilon_2$. In particular, this would imply that all stuttering look and say sequences grow at rate (\ref{stuttering Conway constant}).

\newpage

\appendix

\section{Python}\label{appendix::python}
The following Python session uses the \texttt{look\_and\_say} module, available at \newline \url{https://pypi.org/project/look-and-say/}. 
This module can be used to study the chemical properties of many variations of look and say sequences. When the chemistry is large, as it is for our stuttering look and say sequences, some of the methods (such as \texttt{generate\_elements()} and \texttt{get\_char\_poly()}) are time intensive. 

{\small
\begin{verbatim}
>>> import look_and_say as ls 
>>> 
>>> # Define the stuttering rule d^n -> n^n d
>>> def stutter(n, d):
...     return (str(n) * n) + d
... 
>>> # Create a LookAndSay object using the stuttering rule
>>> stut_las = ls.LookAndSay(stutter)
>>> 
>>> # Check that the object creates the correct sequence
>>> stut_las.generate_sequence(seed='0', terms=6)
>>> stut_las.get_sequence()
['0', '10', '1110', '333110', '333322110', '4444322222110']
>>> 
>>> # Define a split function corresponding to Theorem 2
>>> sff = ls.SplitFuncFactory()
>>> sff.declare_split_after('0')
>>> split = sff.get_split()
>>> 
>>> # Use a Chemistry object to generate common elements from seed 0
>>> chem = ls.Chemistry(stut_las, split)
>>> chem.generate_elements('0')
>>> 
>>> # Show the total number of common elements
>>> len(chem.get_elements())
714
>>> # Order the elements according to their relative abundance 
>>> chem.order_elements('abundance')
>>>
>>> # Use the periodic table to see properties of elements
>>> periodic_table = chem.get_periodic_table()
>>> periodic_table['E1']
{'string': '10', 'abundance': 27.8158567, 'decay': [E2]}
>>> 
>>> # Show lambda and the characteristic polynomial
>>> chem.get_dom_eigenvalue()
1.4453300117263481
>>> chem.get_char_poly()
lambda**283*(lambda - 1)*(lambda + 1)*(lambda**2 + 1)*(lambda**4 + 1)*...
\end{verbatim}}

\newpage
\pagestyle{empty}

\section{The periodic table}\label{appendix::periodic table} 
\enlargethispage{5\baselineskip}
To save space in the following table, for each $i$ we set
$f_i=e_1^9e_i$,\quad $g_i=e_1^{10}e_i$,\quad $h_i=e_{532}f_{359}^{16}f_{i}$,\quad and\quad $k=f_{237}f_{133}f_{152}$.
{\tiny
\[\begin{array}{|c|l|}
    \hline
    n & \text{decay of }e_n \\\hline
    1 & e_{2}\\
    2 & e_{3}\\
    3 & e_{4}\\
    4 & e_{5}\\
    5 & e_{6}\\
    6 & e_{7}\\
    7 & e_{8}\\
    8 & e_{9}\\
    9 & e_{10}\\
    10 & e_{11}\\
    11 & e_{13}\\
    12 & e_{15}\\
    13 & e_{16}\\
    14 & e_{18}\\
    15 & e_{19}\\
    16 & e_{20}f_{17}\\
    17 & e_{21}\\
    18 & e_{22}\\
    19 & e_{23}\\
    20 & e_{25}f_{12}\\
    21 & e_{28}\\
    22 & e_{29}\\
    23 & e_{30}\\
    24 & e_{31}f_{27}\\
    25 & e_{33}f_{14}\\
    26 & e_{34}f_{35}f_{12}\\
    27 & e_{36}\\
    28 & e_{37}\\
    29 & e_{38}\\
    30 & e_{39}\\
    31 & e_{41}f_{12}\\
    32 & e_{42}\\
    33 & g_{45}\\
    34 & e_{46}f_{12}\\
    35 & e_{40}f_{32}f_{14}\\
    36 & e_{49}\\
    37 & e_{51}\\
    38 & e_{53}\\
    39 & e_{54}\\
    40 & e_{55}\\
    41 & e_{50}f_{47}f_{14}\\
    42 & e_{56}f_{24}\\
    43 & e_{57}\\
    44 & e_{58}\\
    45 & e_{48}f_{24}\\
    46 & e_{59}f_{44}f_{14}\\
    47 & e_{60}\\
    48 & e_{61}\\
    49 & e_{62}\\
    50 & e_{65}\\
    51 & e_{66}\\
    52 & e_{67}\\
    53 & e_{64}f_{68}\\
    54 & e_{69}\\
    55 & e_{71}\\
    56 & e_{72}\\
    57 & e_{75}\\
    58 & e_{73}f_{24}\\
    59 & e_{77}\\
    60 & e_{78}\\
    61 & e_{80}\\
    62 & e_{83}\\
    63 & e_{84}\\
    64 & e_{81}f_{12}\\
    65 & e_{86}\\
    66 & e_{85}\\
    67 & e_{88}\\
    68 & e_{89}\\
    69 & e_{90}\\
    70 & e_{92}\\
    71 & e_{93}\\
    72 & e_{95}\\
    73 & e_{96}\\
    74 & e_{98}\\
    75 & e_{99}f_{82}\\
    76 & e_{100}f_{12}\\
    77 & e_{101}\\
    78 & e_{102}f_{27}\\
    79 & e_{91}f_{94}\\
    80 & e_{103}\\
    81 & e_{79}f_{32}f_{14}\\
    82 & e_{104}\\
    83 & e_{105}\\
    84 & e_{43}f_{70}\\
    \hline
\end{array}
\begin{array}{|c|l|}
    \hline
    n & \text{decay of }e_n \\\hline
    85 & e_{106}\\
    86 & e_{52}f_{63}\\
    87 & e_{108}\\
    88 & e_{109}\\
    89 & e_{111}\\
    90 & e_{112}\\
    91 & e_{76}f_{107}f_{44}f_{14}\\
    92 & e_{43}f_{97}\\
    93 & e_{116}\\
    94 & e_{117}\\
    95 & e_{119}\\
    96 & e_{120}\\
    97 & e_{124}f_{17}\\
    98 & e_{125}\\
    99 & e_{52}f_{113}f_{114}\\
    100 & e_{122}f_{123}f_{118}f_{14}\\
    101 & e_{127}\\
    102 & e_{128}f_{12}\\
    103 & e_{129}\\
    104 & e_{131}\\
    105 & e_{132}\\
    106 & e_{134}\\
    107 & e_{135}\\
    108 & e_{136}f_{126}\\
    109 & e_{137}\\
    110 & e_{138}\\
    111 & e_{139}\\
    112 & e_{76}f_{140}\\
    113 & e_{141}\\
    114 & e_{48}f_{130}\\
    115 & e_{142}\\
    116 & e_{52}f_{70}\\
    117 & e_{143}\\
    118 & e_{145}\\
    119 & e_{146}\\
    120 & e_{147}\\
    121 & e_{144}f_{63}\\
    122 & e_{149}f_{12}\\
    123 & e_{150}\\
    124 & e_{43}f_{115}f_{12}\\
    125 & e_{151}\\
    126 & e_{153}\\
    127 & e_{155}\\
    128 & e_{154}f_{47}f_{14}\\
    129 & e_{110}f_{87}\\
    130 & e_{157}\\
    131 & e_{156}\\
    132 & e_{158}\\
    133 & e_{159}\\
    134 & e_{160}\\
    135 & e_{161}\\
    136 & e_{162}\\
    137 & e_{163}\\
    138 & e_{79}f_{32}f_{121}\\
    139 & e_{164}\\
    140 & e_{165}\\
    141 & e_{43}f_{166}\\
    142 & e_{167}\\
    143 & e_{170}\\
    144 & e_{172}\\
    145 & e_{173}f_{24}\\
    146 & e_{171}f_{87}\\
    147 & e_{174}\\
    148 & e_{175}\\
    149 & e_{20}f_{17}f_{148}f_{14}\\
    150 & e_{176}\\
    151 & e_{177}\\
    152 & e_{178}\\
    153 & e_{179}\\
    154 & e_{110}f_{133}f_{152}\\
    155 & e_{180}\\
    156 & e_{182}\\
    157 & e_{183}\\
    158 & e_{184}\\
    159 & e_{168}f_{24}\\
    160 & e_{185}\\
    161 & e_{186}\\
    162 & e_{187}\\
    163 & e_{188}\\
    164 & e_{189}\\
    165 & e_{190}f_{17}\\
    166 & e_{191}\\
    167 & e_{193}\\
    168 & e_{194}\\
    \hline
\end{array}
\begin{array}{|c|l|}
    \hline
    n & \text{decay of }e_n \\\hline
    169 & e_{201}f_{200}\\
    170 & e_{192}f_{70}\\
    171 & e_{198}f_{202}\\
    172 & e_{204}f_{203}\\
    173 & e_{205}\\
    174 & e_{206}f_{87}\\
    175 & e_{207}f_{24}\\
    176 & e_{208}\\
    177 & e_{196}f_{195}f_{17}\\
    178 & e_{169}f_{63}\\
    179 & e_{209}\\
    180 & e_{210}\\
    181 & e_{211}f_{12}\\
    182 & e_{213}f_{17}\\
    183 & e_{50}f_{212}\\
    184 & e_{214}\\
    185 & e_{217}\\
    186 & e_{218}\\
    187 & e_{219}f_{215}\\
    188 & e_{220}f_{47}f_{197}\\
    189 & e_{222}f_{17}\\
    190 & e_{223}f_{12}\\
    191 & e_{43}f_{224}\\
    192 & e_{225}\\
    193 & e_{226}\\
    194 & e_{227}\\
    195 & e_{48}f_{221}\\
    196 & e_{52}f_{199}\\
    197 & e_{228}f_{17}\\
    198 & e_{91}f_{12}\\
    199 & e_{229}\\
    200 & e_{230}\\
    201 & e_{231}\\
    202 & e_{40}f_{32}f_{121}\\
    203 & e_{40}f_{232}f_{82}\\
    204 & e_{233}\\
    205 & e_{235}\\
    206 & e_{236}\\
    207 & e_{238}\\
    208 & e_{239}\\
    209 & e_{240}\\
    210 & e_{241}f_{97}\\
    211 & e_{20}f_{242}\\
    212 & e_{245}\\
    213 & e_{52}f_{115}f_{12}\\
    214 & e_{246}f_{17}\\
    215 & e_{247}\\
    216 & e_{250}\\
    217 & e_{181}f_{251}\\
    218 & e_{252}\\
    219 & e_{253}f_{12}\\
    220 & e_{20}f_{17}f_{148}f_{152}\\
    221 & e_{255}\\
    222 & e_{256}f_{12}\\
    223 & e_{257}f_{44}f_{14}\\
    224 & e_{258}\\
    225 & e_{260}f_{259}\\
    226 & e_{261}\\
    227 & e_{262}\\
    228 & e_{169}f_{115}f_{12}\\
    229 & e_{43}f_{63}\\
    230 & e_{263}f_{82}\\
    231 & e_{264}\\
    232 & e_{266}\\
    233 & e_{265}\\
    234 & e_{267}\\
    235 & e_{268}\\
    236 & e_{269}f_{32}f_{121}\\
    237 & e_{270}\\
    238 & e_{271}\\
    239 & e_{272}\\
    240 & e_{273}f_{17}\\
    241 & e_{274}\\
    242 & e_{277}\\
    243 & e_{276}\\
    244 & g_{244}\\
    245 & e_{278}f_{126}\\
    246 & e_{249}f_{248}f_{275}f_{12}\\
    247 & e_{279}\\
    248 & e_{282}\\
    249 & e_{181}f_{281}\\
    250 & e_{43}f_{283}\\
    251 & e_{284}\\
    252 & e_{285}\\
    \hline
\end{array}
\begin{array}{|c|l|}
    \hline
    n & \text{decay of }e_n \\\hline
    253 & e_{81}f_{237}f_{133}f_{14}\\
    254 & e_{48}f_{286}\\
    255 & e_{287}\\
    256 & e_{288}f_{133}f_{14}\\
    257 & e_{289}\\
    258 & e_{43}f_{216}\\
    259 & e_{291}\\
    260 & e_{290}\\
    261 & e_{52}f_{216}\\
    262 & e_{292}\\
    263 & e_{234}f_{113}f_{114}\\
    264 & e_{64}f_{294}\\
    265 & e_{295}\\
    266 & e_{168}f_{130}\\
    267 & e_{296}\\
    268 & e_{298}\\
    269 & e_{76}f_{297}f_{94}\\
    270 & e_{299}\\
    271 & e_{300}\\
    272 & e_{301}\\
    273 & e_{303}f_{302}f_{12}\\
    274 & e_{181}f_{304}\\
    275 & e_{280}f_{254}f_{14}\\
    276 & e_{309}\\
    277 & e_{306}f_{308}\\
    278 & e_{310}\\
    279 & e_{311}\\
    280 & e_{43}f_{199}\\
    281 & e_{313}\\
    282 & e_{314}\\
    283 & e_{315}\\
    284 & e_{243}f_{74}^{10}f_{316}\\
    285 & e_{312}f_{97}\\
    286 & e_{31}f_{318}\\
    287 & e_{50}f_{305}\\
    288 & e_{319}\\
    289 & e_{320}\\
    290 & e_{322}\\
    291 & e_{73}f_{321}\\
    292 & e_{81}f_{237}f_{87}\\
    293 & e_{323}\\
    294 & e_{324}\\
    295 & e_{325}\\
    296 & e_{326}\\
    297 & e_{327}f_{44}f_{14}\\
    298 & e_{328}f_{87}\\
    299 & e_{234}f_{63}\\
    300 & e_{329}\\
    301 & e_{330}\\
    302 & e_{332}f_{317}f_{14}\\
    303 & e_{331}\\
    304 & e_{333}\\
    305 & e_{334}f_{17}\\
    306 & e_{335}\\
    307 & e_{337}f_{133}f_{14}\\
    308 & e_{339}\\
    309 & e_{338}\\
    310 & e_{336}f_{293}\\
    311 & e_{340}\\
    312 & e_{8}f_{341}\\
    313 & e_{243}f_{74}^{10}f_{342}\\
    314 & e_{343}\\
    315 & e_{344}\\
    316 & e_{345}f_{244}\\
    317 & e_{346}f_{45}\\
    318 & e_{347}\\
    319 & e_{348}\\
    320 & e_{349}\\
    321 & e_{351}\\
    322 & e_{350}\\
    323 & e_{353}\\
    324 & e_{354}\\
    325 & e_{79}f_{355}\\
    326 & e_{356}\\
    327 & e_{357}\\
    328 & e_{358}\\
    329 & e_{352}f_{87}\\
    330 & e_{360}f_{17}\\
    331 & e_{361}\\
    332 & e_{362}f_{12}\\
    333 & e_{243}f_{74}^{10}f_{363}\\
    334 & e_{364}\\
    335 & e_{366}f_{26}^{175}f_{365}f_{293}\\
    \hline
\end{array}\]

\[\begin{array}{|c|l|}
    \hline
    n & \text{decay of }e_n \\\hline
    336 & e_{367}f_{12}\\
    337 & e_{368}\\
    338 & e_{52}f_{369}\\
    339 & e_{370}\\
    340 & e_{372}\\
    341 & e_{373}\\
    342 & e_{376}\\
    343 & e_{374}\\
    344 & e_{375}\\
    345 & e_{377}\\
    346 & e_{378}\\
    347 & e_{379}\\
    348 & e_{249}f_{248}f_{63}\\
    349 & e_{380}\\
    350 & e_{381}\\
    351 & e_{382}\\
    352 & e_{383}\\
    353 & e_{385}\\
    354 & e_{386}\\
    355 & e_{387}\\
    356 & e_{388}\\
    357 & e_{389}\\
    358 & e_{390}f_{32}f_{121}\\
    359 & e_{391}f_{12}\\
    360 & e_{384}f_{115}f_{12}\\
    361 & e_{393}\\
    362 & e_{50}f_{254}f_{14}\\
    363 & e_{394}\\
    364 & e_{395}\\
    365 & e_{34}f_{396}f_{371}f_{12}\\
    366 & e_{397}f_{12}\\
    367 & e_{110}f_{133}f_{14}\\
    368 & e_{398}\\
    369 & e_{280}f_{195}f_{17}\\
    370 & e_{234}f_{216}\\
    371 & e_{399}f_{118}f_{14}\\
    372 & e_{400}\\
    373 & e_{401}\\
    374 & e_{402}\\
    375 & e_{403}f_{17}\\
    376 & e_{404}\\
    377 & e_{405}\\
    378 & e_{406}\\
    379 & e_{407}\\
    380 & e_{408}\\
    381 & e_{409}\\
    382 & e_{50}f_{47}f_{197}\\
    383 & e_{410}f_{32}f_{121}\\
    384 & e_{412}f_{411}\\
    385 & e_{413}\\
    386 & e_{414}\\
    387 & e_{415}f_{126}\\
    388 & e_{416}\\
    389 & e_{417}\\
    390 & e_{418}f_{94}\\
    391 & e_{419}f_{44}f_{14}\\
    392 & e_{196}f_{254}f_{14}\\
    393 & e_{76}f_{107}f_{420}\\
    394 & e_{421}\\
    395 & e_{422}\\
    396 & e_{423}f_{12}\\
    397 & e_{85}f_{148}f_{14}\\
    398 & e_{192}f_{63}\\
    399 & e_{424}\\
    400 & e_{425}\\
    401 & e_{426}\\
    402 & e_{428}\\
    403 & e_{429}f_{12}\\
    404 & e_{430}f_{427}\\
    405 & e_{392}f_{431}\\
    406 & e_{432}\\
    407 & e_{433}\\
    408 & e_{434}\\
    409 & e_{122}f_{123}f_{118}f_{435}\\
    410 & e_{436}f_{94}\\
    411 & e_{439}\\
    412 & e_{438}\\
    413 & e_{440}\\
    414 & e_{441}\\
    415 & e_{442}\\
    416 & e_{76}f_{107}f_{44}f_{197}\\
    417 & e_{444}\\
    418 & e_{443}f_{107}f_{44}f_{14}\\
    419 & e_{445}\\
    420 & e_{446}f_{17}\\
    421 & e_{447}f_{47}f_{197}\\
    422 & e_{154}f_{305}\\
    423 & e_{448}f_{32}f_{14}\\
    424 & e_{449}\\
    425 & e_{450}\\
    426 & e_{451}\\
    427 & e_{453}\\
    428 & e_{452}f_{47}f_{197}\\
    429 & e_{52}f_{113}f_{454}f_{317}f_{14}\\
    430 & e_{392}f_{12}\\
    \hline
\end{array}
\begin{array}{|c|l|}
    \hline
    n & \text{decay of }e_n \\\hline
    431 & e_{455}\\
    432 & e_{456}\\
    433 & e_{457}\\
    434 & e_{249}f_{248}f_{97}\\
    435 & e_{144}f_{97}\\
    436 & e_{458}f_{107}f_{44}f_{14}\\
    437 & e_{459}f_{12}\\
    438 & e_{460}\\
    439 & e_{461}\\
    440 & e_{462}\\
    441 & e_{463}\\
    442 & e_{464}\\
    443 & e_{122}f_{371}f_{12}\\
    444 & e_{465}\\
    445 & e_{466}\\
    446 & e_{73}f_{221}\\
    447 & e_{467}f_{35}k\\
    448 & e_{468}\\
    449 & e_{469}\\
    450 & e_{471}\\  
    451 & e_{472}\\
    452 & e_{198}f_{35}k\\
    453 & e_{475}\\
    454 & e_{474}f_{473}f_{12}\\
    455 & e_{470}f_{476}\\
    456 & e_{477}\\
    457 & e_{478}\\
    458 & e_{479}f_{12}\\
    459 & e_{28}f_{148}f_{14}\\
    460 & e_{480}\\
    461 & e_{481}\\
    462 & e_{482}\\
    463 & e_{483}\\
    464 & e_{484}f_{215}\\
    465 & e_{485}\\
    466 & e_{486}\\
    467 & e_{196}f_{254}f_{488}f_{12}\\
    468 & e_{489}\\
    469 & e_{491}\\
    470 & e_{234}f_{275}f_{307}^{14}f_{12}\\
    471 & e_{492}f_{17}\\
    472 & e_{493}f_{47}f_{197}\\
    473 & e_{495}f_{254}f_{14}\\
    474 & e_{496}\\
    475 & e_{494}\\
    476 & e_{497}\\
    477 & e_{498}\\
    478 & e_{499}\\
    479 & e_{500}f_{123}f_{118}f_{14}\\
    480 & e_{20}f_{501}\\
    481 & e_{502}\\
    482 & e_{503}\\
    483 & e_{504}\\
    484 & e_{505}f_{12}\\
    485 & e_{490}f_{97}\\
    486 & e_{506}\\
    487 & e_{507}f_{12}\\
    488 & e_{508}f_{254}f_{14}\\
    489 & e_{509}\\
    490 & e_{510}\\
    491 & e_{511}\\
    492 & e_{512}f_{12}\\
    493 & e_{352}f_{133}f_{152}\\
    494 & e_{513}\\
    495 & e_{515}\\
    496 & e_{514}\\
    497 & e_{516}\\
    498 & e_{517}\\
    499 & e_{518}\\
    500 & e_{20}f_{519}f_{12}\\
    501 & e_{521}f_{70}\\
    502 & e_{520}\\
    503 & e_{522}\\
    504 & e_{523}\\
    505 & e_{171}f_{133}f_{14}\\
    506 & e_{525}\\
    507 & e_{524}f_{526}f_{47}f_{14}\\
    508 & e_{169}f_{199}\\
    509 & e_{527}\\
    510 & e_{528}\\
    511 & e_{529}\\
    512 & e_{181}f_{531}f_{530}f_{47}f_{14}\\
    513 & e_{533}\\
    514 & e_{336}f_{215}\\
    515 & e_{234}f_{166}\\
    516 & h_{534}\\
    517 & e_{535}\\
    518 & e_{536}\\
    519 & e_{21}f_{148}f_{14}\\
    520 & e_{537}f_{47}f_{197}\\
    521 & e_{306}f_{200}\\
    522 & e_{538}\\
    523 & e_{539}f_{17}\\
    524 & e_{540}f_{45}\\
    525 & e_{541}\\
    \hline
\end{array}
\begin{array}{|c|l|}
    \hline
    n & \text{decay of }e_n \\\hline
    526 & e_{144}f_{199}\\
    527 & e_{312}f_{70}\\
    528 & e_{543}f_{542}\\
    529 & e_{544}\\
    530 & e_{545}f_{199}\\
    531 & e_{546}\\
    532 & e_{547}f_{12}\\
    533 & e_{548}\\
    534 & e_{549}\\
    535 & e_{550}\\
    536 & e_{551}\\
    537 & e_{81}k\\
    538 & e_{552}f_{17}\\
    539 & e_{554}f_{12}\\
    540 & e_{555}\\
    541 & e_{384}f_{97}\\
    542 & e_{557}\\
    543 & e_{556}f_{12}\\
    544 & e_{558}f_{17}\\
    545 & e_{560}\\
    546 & e_{243}f_{74}^{10}f_{559}\\
    547 & e_{553}f_{133}f_{14}\\
    548 & e_{561}f_{17}\\
    549 & e_{562}\\
    550 & e_{563}\\
    551 & e_{564}f_{17}\\
    552 & e_{181}f_{565}f_{12}\\
    553 & e_{566}\\
    554 & e_{181}f_{567}f_{133}f_{14}\\
    555 & e_{568}\\
    556 & e_{569}\\
    557 & e_{570}\\
    558 & e_{241}f_{115}f_{12}\\
    559 & e_{572}\\
    560 & e_{571}\\
    561 & e_{490}f_{275}f_{12}\\
    562 & e_{573}f_{437}^{17}f_{574}f_{17}\\
    563 & e_{181}f_{576}\\
    564 & e_{181}f_{575}f_{275}f_{12}\\
    565 & e_{578}f_{47}f_{14}\\
    566 & e_{577}\\
    567 & e_{579}f_{580}\\
    568 & e_{581}\\
    569 & e_{20}f_{582}\\
    570 & e_{583}\\
    571 & e_{584}\\
    572 & e_{585}\\
    573 & e_{586}f_{12}\\
    574 & e_{587}f_{12}\\
    575 & e_{588}\\
    576 & e_{589}\\
    577 & e_{490}f_{63}\\
    578 & e_{243}f_{74}^{10}f_{590}f_{199}\\
    579 & e_{243}f_{74}^{10}f_{591}\\
    580 & e_{593}\\
    581 & e_{592}\\
    582 & e_{595}\\
    583 & e_{594}f_{47}f_{197}\\
    584 & e_{596}\\
    585 & e_{597}\\
    586 & e_{598}f_{44}f_{14}\\
    587 & e_{599}f_{14}\\
    588 & e_{243}f_{74}^{10}f_{600}\\
    589 & e_{243}f_{74}^{10}f_{601}f_{17}\\
    590 & e_{602}\\
    591 & e_{603}\\
    592 & e_{605}\\
    593 & e_{604}f_{63}\\
    594 & e_{553}f_{133}f_{152}\\
    595 & e_{606}\\
    596 & e_{609}\\
    597 & e_{392}f_{608}\\
    598 & e_{612}\\
    599 & e_{613}f_{487}^{23}f_{611}f_{610}\\
    600 & e_{614}\\
    601 & e_{345}f_{607}f_{12}\\
    602 & e_{617}\\
    603 & e_{618}\\
    604 & e_{616}\\
    605 & e_{615}\\
    606 & e_{619}\\
    607 & g_{620}f_{14}\\
    608 & e_{622}\\
    609 & e_{621}\\
    610 & e_{625}\\
    611 & e_{626}f_{12}\\
    612 & e_{623}\\
    613 & e_{624}f_{12}\\
    614 & e_{627}\\
    615 & e_{629}\\
    616 & e_{630}\\
    617 & e_{628}\\
    618 & e_{631}\\
    619 & e_{366}f_{26}^{175}f_{365}f_{632}\\
    \hline
\end{array}
\begin{array}{|c|l|}
    \hline
    n & \text{decay of }e_n \\\hline
    620 & g_{633}f_{45}\\
    621 & e_{635}f_{47}f_{197}\\
    622 & e_{470}f_{634}\\
    623 & e_{636}\\
    624 & e_{66}f_{148}f_{14}\\
    625 & e_{638}\\
    626 & e_{524}f_{14}\\
    627 & e_{637}\\
    628 & e_{392}f_{639}\\
    629 & e_{641}\\
    630 & e_{642}\\
    631 & e_{392}f_{640}\\
    632 & e_{643}\\
    633 & g_{644}\\
    634 & e_{649}\\
    635 & e_{646}f_{35}k\\
    636 & e_{647}\\
    637 & e_{392}f_{645}\\
    638 & e_{648}\\
    639 & e_{652}\\
    640 & e_{653}\\
    641 & e_{181}f_{650}\\
    642 & e_{651}\\
    643 & e_{654}\\
    644 & g_{655}\\
    645 & e_{659}f_{47}f_{197}\\
    646 & e_{76}f_{297}f_{12}\\
    647 & e_{656}\\
    648 & e_{658}\\
    649 & e_{657}\\
    650 & e_{663}\\
    651 & e_{662}\\
    652 & e_{470}f_{660}\\
    653 & e_{470}f_{661}\\
    654 & e_{664}\\
    655 & g_{665}\\
    656 & e_{668}\\
    657 & e_{666}\\
    658 & e_{669}\\
    659 & e_{470}f_{667}f_{35}k\\
    660 & e_{672}f_{47}f_{197}\\
    661 & e_{673}\\
    662 & e_{671}\\
    663 & e_{243}f_{74}^{10}f_{670}\\
    664 & e_{674}\\
    665 & g_{675}\\
    666 & e_{677}\\
    667 & e_{678}f_{297}f_{12}\\
    668 & e_{181}f_{575}f_{97}\\
    669 & e_{676}\\
    670 & e_{345}f_{679}\\
    671 & e_{680}f_{47}f_{197}\\
    672 & e_{681}f_{35}k\\
    673 & e_{682}\\
    674 & e_{685}\\
    675 & g_{683}\\
    676 & e_{690}\\
    677 & e_{684}f_{17}\\
    678 & e_{691}f_{12}\\
    679 & g_{686}f_{17}\\
    680 & e_{689}f_{35}k\\
    681 & h_{688}f_{297}f_{12}\\
    682 & h_{687}\\
    683 & g_{692}\\
    684 & e_{694}f_{12}\\
    685 & e_{695}\\
    686 & g_{607}f_{12}\\
    687 & e_{700}\\
    688 & e_{34}f_{371}f_{12}\\
    689 & e_{443}f_{297}f_{12}\\
    690 & e_{693}\\
    691 & e_{50}f_{133}f_{14}\\
    692 & g_{696}\\
    693 & e_{697}\\
    694 & e_{699}f_{133}f_{14}\\
    695 & e_{698}f_{17}\\
    696 & g_{701}\\
    697 & e_{702}\\
    698 & e_{704}f_{12}\\
    699 & e_{703}\\
    700 & e_{705}\\
    701 & g_{708}\\
    702 & e_{711}\\
    703 & e_{712}\\
    704 & e_{713}f_{47}f_{14}\\
    705 & e_{573}f_{17}\\
    706 & g_{679}\\
    707 & g_{706}\\
    708 & g_{707}\\
    709 & e_{714}f_{47}f_{14}\\
    710 & e_{579}f_{709}f_{12}\\
    711 & e_{181}f_{710}f_{17}\\
    712 & e_{241}f_{63}\\
    713 & e_{241}f_{199}\\
    714 & e_{604}f_{199}\\
    \hline
\end{array}\]
}


\section{The polynomial}\label{appendix: the polynomial} 
\enlargethispage{5\baselineskip}
\tiny{\noindent$\lambda^{415}\!-\lambda^{414}\!- 2\lambda^{410}\!+\lambda^{409}\!+\lambda^{408}\!- 2\lambda^{407}\!+\lambda^{406}\!-\lambda^{404}\!- 10\lambda^{402}\!+\lambda^{401}\!- 2\lambda^{400}\!+ 4\lambda^{399}\!- 152\lambda^{397}\!+ 172\lambda^{396}\!+ 14\lambda^{395}\!+ 24\lambda^{394}\!+ 3\lambda^{393}\!+ 164\lambda^{392}\!- 156\lambda^{391}\!- 23\lambda^{390}\!+ 139\lambda^{389}\!- 235\lambda^{388}\!+ 201\lambda^{387}\!+ 167\lambda^{386}\!+ 23\lambda^{385}\!+ 27\lambda^{384}\!- 2349\lambda^{383}\!- 1615\lambda^{382}\!+ 2198\lambda^{381}\!+ 1178\lambda^{380}\!- 2152\lambda^{379}\!+ 5149\lambda^{378}\!+ 508\lambda^{377}\!+ 1331\lambda^{376}\!- 419\lambda^{375}\!+ 417\lambda^{374}\!- 1561\lambda^{373}\!+ 3735\lambda^{372}\!+ 8875\lambda^{371}\!- 16260\lambda^{370}\!+ 11579\lambda^{369}\!- 10976\lambda^{368}\!+ 1496\lambda^{367}\!- 4706\lambda^{366}\!- 9123\lambda^{365}\!- 1915\lambda^{364}\!+ 18287\lambda^{363}\!+ 45625\lambda^{362}\!- 71399\lambda^{361}\!- 17446\lambda^{360}\!- 19863\lambda^{359}\!- 70156\lambda^{358}\!+ 90360\lambda^{357}\!- 42652\lambda^{356}\!+ 71132\lambda^{355}\!+ 46501\lambda^{354}\!+ 214999\lambda^{353}\!- 79007\lambda^{352}\!- 151659\lambda^{351}\!+ 257677\lambda^{350}\!- 446037\lambda^{349}\!+ 292253\lambda^{348}\!- 713468\lambda^{347}\!+ 436555\lambda^{346}\!+ 364420\lambda^{345}\!+ 338653\lambda^{344}\!- 134998\lambda^{343}\!- 10046\lambda^{342}\!+ 22201\lambda^{341}\!- 438631\lambda^{340}\!+ 577203\lambda^{339}\!- 177704\lambda^{338}\!- 990256\lambda^{337}\!+ 917295\lambda^{336}\!- 2410605\lambda^{335}\!- 287705\lambda^{334}\!- 1611151\lambda^{333}\!+ 3591794\lambda^{332}\!- 323597\lambda^{331}\!+ 1867931\lambda^{330}\!- 3806263\lambda^{329}\!+ 288492\lambda^{328}\!- 4967777\lambda^{327}\!+ 3848436\lambda^{326}\!- 7313096\lambda^{325}\!+ 8412621\lambda^{324}\!- 4086405\lambda^{323}\!+ 10656364\lambda^{322}\!- 12221741\lambda^{321}\!+ 4691770\lambda^{320}\!- 3104647\lambda^{319}\!+ 9384603\lambda^{318}\!- 12741889\lambda^{317}\!+ 28175225\lambda^{316}\!- 36732683\lambda^{315}\!+ 27549087\lambda^{314}\!- 29113688\lambda^{313}\!+ 27904936\lambda^{312}\!- 64022174\lambda^{311}\!+ 86900035\lambda^{310}\!- 44038451\lambda^{309}\!+ 69272187\lambda^{308}\!- 100228015\lambda^{307}\!+ 101153700\lambda^{306}\!- 112360662\lambda^{305}\!+ 140263944\lambda^{304}\!- 198709702\lambda^{303}\!+ 227106421\lambda^{302}\!- 199066229\lambda^{301}\!+ 158108336\lambda^{300}\!- 177991300\lambda^{299}\!+ 160265366\lambda^{298}\!- 216764982\lambda^{297}\!+ 458594331\lambda^{296}\!- 514396955\lambda^{295}\!+ 623848863\lambda^{294}\!- 478106150\lambda^{293}\!+ 319570262\lambda^{292}\!- 449486144\lambda^{291}\!+ 242806151\lambda^{290}\!- 503482313\lambda^{289}\!+ 537025658\lambda^{288}\!- 117520757\lambda^{287}\!+ 205010170\lambda^{286}\!+ 216142705\lambda^{285}\!- 69286579\lambda^{284}\!- 249935826\lambda^{283}\!- 14860455\lambda^{282}\!- 56994233\lambda^{281}\!+ 694547\lambda^{280}\!+ 345346471\lambda^{279}\!- 469219392\lambda^{278}\!+ 594365027\lambda^{277}\!- 1037912862\lambda^{276}\!+ 1069314068\lambda^{275}\!- 1277383331\lambda^{274}\!+ 2290966254\lambda^{273}\!- 3704443101\lambda^{272}\!+ 4343428926\lambda^{271}\!- 4028145929\lambda^{270}\!+ 3680122414\lambda^{269}\!- 2315094078\lambda^{268}\!+ 570793820\lambda^{267}\!- 2629080969\lambda^{266}\!+ 6410285252\lambda^{265}\!- 7133144189\lambda^{264}\!+ 6733705466\lambda^{263}\!- 5378790644\lambda^{262}\!+ 8228210949\lambda^{261}\!- 8942185731\lambda^{260}\!+ 5979630418\lambda^{259}\!- 7905989864\lambda^{258}\!+ 9737011738\lambda^{257}\!- 9911610982\lambda^{256}\!+ 9973159722\lambda^{255}\!- 14957013413\lambda^{254}\!+ 14890985003\lambda^{253}\!- 10726321400\lambda^{252}\!+ 15125054942\lambda^{251}\!- 20529308050\lambda^{250}\!+ 17284779230\lambda^{249}\!- 15931854028\lambda^{248}\!+ 12442209687\lambda^{247}\!- 23517231728\lambda^{246}\!+ 22425011617\lambda^{245}\!- 17938407328\lambda^{244}\!+ 37345675143\lambda^{243}\!- 46891616723\lambda^{242}\!+ 39563694151\lambda^{241}\!- 20239106041\lambda^{240}\!+ 25060340748\lambda^{239}\!- 43257854036\lambda^{238}\!+ 52510645950\lambda^{237}\!- 49049636561\lambda^{236}\!+ 53510422084\lambda^{235}\!- 49930847956\lambda^{234}\!+ 49528038601\lambda^{233}\!- 65436262565\lambda^{232}\!+ 69787936146\lambda^{231}\!- 62639300626\lambda^{230}\!+ 61154652466\lambda^{229}\!- 103295120449\lambda^{228}\!+ 150716638429\lambda^{227}\!- 161589047937\lambda^{226}\!+ 68355002194\lambda^{225}\!- 71534955143\lambda^{224}\!+ 108040997051\lambda^{223}\!- 70767774696\lambda^{222}\!+ 52614630562\lambda^{221}\!- 20391846659\lambda^{220}\!- 14432367525\lambda^{219}\!- 44881321887\lambda^{218}\!+ 56948736961\lambda^{217}\!- 155089185437\lambda^{216}\!+ 301707147402\lambda^{215}\!- 241030438261\lambda^{214}\!+ 209243008202\lambda^{213}\!- 129455951837\lambda^{212}\!- 22600217909\lambda^{211}\!+ 169593728583\lambda^{210}\!- 233692316096\lambda^{209}\!+ 257322864009\lambda^{208}\!- 243574029523\lambda^{207}\!+ 149000586795\lambda^{206}\!+ 6984614817\lambda^{205}\!- 144155421524\lambda^{204}\!+ 227532222220\lambda^{203}\!- 267555275442\lambda^{202}\!+ 439293169044\lambda^{201}\!- 480642841570\lambda^{200}\!+ 24888706756\lambda^{199}\!+ 390863025544\lambda^{198}\!- 669336450272\lambda^{197}\!+ 660787478734\lambda^{196}\!- 393807654095\lambda^{195}\!+ 118652462360\lambda^{194}\!+ 177904574224\lambda^{193}\!- 195090491911\lambda^{192}\!+ 242407763338\lambda^{191}\!- 434438066224\lambda^{190}\!+ 555624004927\lambda^{189}\!- 577323055405\lambda^{188}\!+ 73441508664\lambda^{187}\!+ 459487464873\lambda^{186}\!- 619671087133\lambda^{185}\!+ 785938357606\lambda^{184}\!- 691599089099\lambda^{183}\!+ 584717564562\lambda^{182}\!- 229276164960\lambda^{181}\!- 435302270582\lambda^{180}\!+ 1018087980286\lambda^{179}\!- 1151064030605\lambda^{178}\!+ 860147866187\lambda^{177}\!- 903396131713\lambda^{176}\!+ 467921483551\lambda^{175}\!+ 48055410710\lambda^{174}\!- 2287054714\lambda^{173}\!+ 657315271193\lambda^{172}\!- 845023767140\lambda^{171}\!+ 365819144460\lambda^{170}\!- 327974853531\lambda^{169}\!- 921734028252\lambda^{168}\!+ 2592590065033\lambda^{167}\!- 3217381089048\lambda^{166}\!+ 3454490790540\lambda^{165}\!- 2562807399867\lambda^{164}\!+ 1138110641904\lambda^{163}\!- 1117436518852\lambda^{162}\!+ 1081360677836\lambda^{161}\!+ 159127306370\lambda^{160}\!- 1003017856247\lambda^{159}\!+ 769570550798\lambda^{158}\!- 61453236065\lambda^{157}\!- 1882037626075\lambda^{156}\!+ 3471445808468\lambda^{155}\!- 3820051566185\lambda^{154}\!+ 4964115234476\lambda^{153}\!- 4388748242709\lambda^{152}\!+ 2087614094823\lambda^{151}\!- 1012832380698\lambda^{150}\!- 402147701987\lambda^{149}\!+ 2239666022624\lambda^{148}\!- 3024989898833\lambda^{147}\!+ 2759983321854\lambda^{146}\!- 1245321388801\lambda^{145}\!- 1265344589045\lambda^{144}\!+ 3283087296960\lambda^{143}\!- 3786332193048\lambda^{142}\!+ 4228696334836\lambda^{141}\!- 3540636875839\lambda^{140}\!+ 1660177144413\lambda^{139}\!- 657558027230\lambda^{138}\!- 1393159449378\lambda^{137}\!+ 4479263491228\lambda^{136}\!- 5143403161319\lambda^{135}\!+ 3548391217258\lambda^{134}\!- 2061855337292\lambda^{133}\!+ 608347763793\lambda^{132}\!+ 916702245119\lambda^{131}\!- 1403656888380\lambda^{130}\!+ 895387140470\lambda^{129}\!- 596962056468\lambda^{128}\!- 372593898767\lambda^{127}\!+ 1722815865484\lambda^{126}\!- 2876081368892\lambda^{125}\!+ 4903989961747\lambda^{124}\!- 5184139923078\lambda^{123}\!+ 3827669726146\lambda^{122}\!- 3396401493971\lambda^{121}\!+ 1840634701896\lambda^{120}\!- 57894012086\lambda^{119}\!+ 420269636152\lambda^{118}\!- 931203594538\lambda^{117}\!+ 749838989034\lambda^{116}\!- 1303848098102\lambda^{115}\!+ 1712367225787\lambda^{114}\!- 1079775670031\lambda^{113}\!+ 798922063995\lambda^{112}\!- 358624404143\lambda^{111}\!- 964934047977\lambda^{110}\!+ 1344018008345\lambda^{109}\!- 1201801395764\lambda^{108}\!+ 1414676146093\lambda^{107}\!- 545485170983\lambda^{106}\!- 425781916461\lambda^{105}\!+ 1497191180344\lambda^{104}\!- 2784441854841\lambda^{103}\!+ 2265308201375\lambda^{102}\!- 1439585313026\lambda^{101}\!+ 933888278223\lambda^{100}\!+ 216751307593\lambda^{99}\!- 947345111075\lambda^{98}\!+ 784146068231\lambda^{97}\!- 561073801579\lambda^{96}\!+ 880818469529\lambda^{95}\!- 628605474171\lambda^{94}\!- 63663745594\lambda^{93}\!+ 293310788698\lambda^{92}\!+ 54298932389\lambda^{91}\!- 367105535532\lambda^{90}\!+ 8793190896\lambda^{89}\!- 255360886114\lambda^{88}\!+ 697553443832\lambda^{87}\!- 273126462799\lambda^{86}\!+ 290421822161\lambda^{85}\!- 344310370986\lambda^{84}\!- 341073041946\lambda^{83}\!+ 528392372325\lambda^{82}\!- 101030105134\lambda^{81}\!- 172978077537\lambda^{80}\!+ 288430500385\lambda^{79}\!- 261105405990\lambda^{78}\!+ 144281228493\lambda^{77}\!- 192899310909\lambda^{76}\!+ 68527486502\lambda^{75}\!+ 160630993291\lambda^{74}\!- 15328235046\lambda^{73}\!+ 47965613456\lambda^{72}\!- 179942185501\lambda^{71}\!- 228930512565\lambda^{70}\!+ 281912848387\lambda^{69}\!- 30539724554\lambda^{68}\!+ 238082470425\lambda^{67}\!- 308262713898\lambda^{66}\!+ 239226529191\lambda^{65}\!- 258301267138\lambda^{64}\!+ 149717463671\lambda^{63}\!+ 44246019876\lambda^{62}\!- 132720953601\lambda^{61}\!+ 126942751272\lambda^{60}\!- 141173427985\lambda^{59}\!+ 198633867855\lambda^{58}\!- 286082550789\lambda^{57}\!+ 253927446808\lambda^{56}\!- 203924714198\lambda^{55}\!+ 284882780202\lambda^{54}\!- 231164110858\lambda^{53}\!+ 82138577114\lambda^{52}\!- 107303267293\lambda^{51}\!+ 127779547325\lambda^{50}\!- 16831165349\lambda^{49}\!- 46903599889\lambda^{48}\!+ 51363804233\lambda^{47}\!- 54113529258\lambda^{46}\!+ 33012502618\lambda^{45}\!+ 1116589673\lambda^{44}\!- 46735331889\lambda^{43}\!+ 53707706055\lambda^{42}\!- 27759673158\lambda^{41}\!+ 35942236335\lambda^{40}\!- 58341717530\lambda^{39}\!+ 42389805588\lambda^{38}\!- 14971696112\lambda^{37}\!+ 19275484678\lambda^{36}\!- 22276279446\lambda^{35}\!+ 5986199951\lambda^{34}\!+ 9849902838\lambda^{33}\!- 13940397607\lambda^{32}\!+ 10226428347\lambda^{31}\!- 5814798734\lambda^{30}\!+ 1286441086\lambda^{29}\!+ 3329800195\lambda^{28}\!- 5980937407\lambda^{27}\!+ 5277989565\lambda^{26}\!- 4977690430\lambda^{25}\!+ 3901846285\lambda^{24}\!- 1623245512\lambda^{23}\!+ 600047358\lambda^{22}\!+ 762053872\lambda^{21}\!- 2190617631\lambda^{20}\!+ 2464733477\lambda^{19}\!- 1556918104\lambda^{18}\!+ 694596945\lambda^{17}\!- 576963184\lambda^{16}\!+ 436526589\lambda^{15}\!- 151831670\lambda^{14}\!- 2902204\lambda^{13}\!+ 52238285\lambda^{12}\!- 79516331\lambda^{11}\!+ 52316396\lambda^{10}\!- 20597389\lambda^{9}\!+ 10506178\lambda^{8}\!- 11260431\lambda^{7}\!+ 7817292\lambda^{6}\!- 1519980\lambda^{5}\!- 711051\lambda^{4}\!- 79915\lambda^{3}\!- 6281\lambda^{2}\!- 407\lambda\!- 11$}

\normalsize{

}


\vfill\eject

\end{document}